\documentclass[letterpaper,12pt]{amsart}

\usepackage{hyperref}

\usepackage{color}
\usepackage{ulem,ifthen,xcolor,xkeyval,pdfcolmk}

\newcommand{\pa}{\partial}

\usepackage[abbrev]{amsrefs}
\usepackage{amssymb,amsmath,amsthm}

\newtheorem{thm}{Theorem}[section]
\newtheorem{prop}[thm]{Proposition}

\newtheorem{lem}[thm]{Lemma}

\newtheorem{remark}[thm]{Remark}

\newcommand{\supp}{\mathop{\mathrm{supp}}\nolimits}

\newcommand{\eps}{\varepsilon}

\def \<{\langle}
\def \>{\rangle}

\def \R{\mathbb R}

\def \H{{\cal H}}

\def \H^0{{\cal H}^0 or}

\def \p{\partial}
\def \n{\nabla}
\def \beq{\begin{equation}}
\def \eeq{\end{equation}}

\def \n{\nabla}

\def \eref{\eqref}

\newcommand{\tr}{\mathop{\mathrm{tr}}}

\newcommand{\Vol}{\operatorname{Vol}}
\newcommand{\cN}{\mathcal{N}}

\newtheorem{cor}[thm]{Corollary}

\begin{document}



\title[Heat trace for drifting Laplacian and Schr\"odinger operators]{The heat trace for the drifting Laplacian and Schr\"odinger operators on manifolds}


\author{Nelia Charalambous}

\author{Julie Rowlett}


\begin{abstract} We study the heat trace for both the drifting Laplacian  as well as Schr\"odinger operators on compact Riemannian manifolds. In the case of a finite regularity potential or weight function, we prove the existence of a partial (six term) asymptotic expansion of the heat trace for small times as well as a suitable remainder estimate.  We also demonstrate that the more precise asymptotic behavior of the remainder is determined by and conversely distinguishes higher (Sobolev) regularity on the potential or weight function.  In the case of a smooth weight function, we determine the full asymptotic expansion of the heat trace for the drifting Laplacian for small times. We then use the heat trace to study the asymptotics of the eigenvalue counting function. In both cases the Weyl law coincides with the Weyl law for the Riemannian manifold with the standard Laplace-Beltrami operator.  We conclude by demonstrating isospectrality results for the drifting Laplacian on compact manifolds.
\end{abstract}

\maketitle

\section{Introduction}
Heat equation methods are incredibly useful; see for example, {\it The Ubiquitous Heat Kernel},
\cite{tubhk}.  One of the most standard ways to prove the spectral theorem for the Laplace operator on a Riemannian manifold is via the associated heat semi-group.  Moreover, the heat equation provides a connection between probability theory and analysis.  As noted in Kac's famous paper, one can thereby demonstrate analytical results using probabilistic methods \cite{kac}.  More recently, van den Berg and Srisatkunarajah used probabilistic methods to demonstrate results for the short time asymptotic behavior of the heat trace on polygonal domains \cite{vdbs}.

The heat trace of a Schr\"odinger operator over a manifold reflects many of the geometric quantities of the manifold intertwined with the potential term. It is also used to determined the rate at which the eigenvalues of the operator tend to infinity. In this article we focus on short time asymptotic expansions of the heat trace for Schr\"odinger operators on manifolds with an irregular potential, as well as the short time asymptotic expansions of the heat trace for drifting Laplacians on weighted manifolds.

To be more precise, we consider $(M,g)$, a smooth, compact, Riemannian manifold of dimension $n$.  The Laplace operator, $\Delta$,  is determined by the Riemannian metric, with
$$\Delta := -\sum_{i,j=1} ^n \frac{1}{\sqrt{\det(g)}} \pa_i g^{ij} \sqrt{\det(g)} \pa_j.$$
In this article, we will consider the Schr\"odinger operator,
$$\Delta_V := \Delta + V,$$
where the potential satisfies $V \in L^\infty (M)$.  We also consider the drifting Laplacian,
$$\Delta_f := \Delta + \nabla f \cdot \nabla.$$
Above, we assume that $f$ is a function on $M$ which satisfies $\nabla f \in L^\infty(M)$ and $\Delta f \in L^\infty (M)$.  The operator $\Delta_f$ is also known as the weighted Laplacian or Bakry-\'Emery Laplacian.  The Schr\"odinger operator is a self-adjoint operator which maps $H^2 (M, g) \to L^2 (M,g)$.  The drifting Laplacian is also a self-adjoint operator with respect to the weighted volume measure
$$d\mu_f = e^{-f} d\mu,$$
where $d\mu$ is the volume measure given by the Riemannian metric.

Our first result is the existence of a small time asymptotic expansion for the trace of the heat kernels associated to $\Delta_V$.  We determine the first few coefficients in this expansion as well as a remainder estimate of order $ o(t^{2-n/2})$ whenever the potential $V$ is in $L^\infty$. Let $e^{-t\Delta_V}$ denote the heat semigroup for the operator $\Delta+V$, and $e^{-t\Delta}$ denote the heat semigroup for the operator $\Delta$.  We note that the operator $\Delta + V$ is a self-adjoint operator which maps $H^2 (M,g) \to L^2(M,g)$.

\begin{prop} \label{prophttr}Let $(M, g)$ be a smooth, compact, Riemannian manifold of dimension $n$.  Let $V \in L^\infty (M)$.  Then $e^{-t \Delta_V}$ is trace class and has a short time asymptotic expansion as $t \downarrow 0$ given by
\begin{equation*}
\begin{split}
\tr e^{-t \Delta_V} = & \frac{\Vol (M)}{(4\pi t)^{n/2}} + \frac{t}{(4\pi t)^{n/2}} \int_M \left( \frac{1}{6} \, K(z)   - V(z)  \right) \;d\mu (z) \\
& +  \frac{t^2}{(4\pi t)^{n/2}}  \left[a_{0,2} + \frac{1}{2} \|V\|_{L^2} ^2 - \frac{1}{6} \int_M  K(z) V(z)  \;d\mu (z)  \right] + o(t^{2-n/2}),
\end{split}
\end{equation*}
where $K(z)$ denotes the scalar curvature of $(M,g)$, and $a_{0,2}$  is the coefficient of $(4\pi t)^{-n/2}t^{2}$
in the short time asymptotic expansion of $e^{-t \Delta}$ given in \eref{a0_i}.
\end{prop}
We prove Proposition \ref{prophttr} in Section \ref{s-heat}. In Corollary \ref{httr_drift} we will see that that a similar expansion also holds for $\Delta_f$ whenever $\Delta f,\n f \in L^\infty$.

Recall that the spectrum, $\sigma(H)$, of a self adjoint operator $H$ on a Hilbert space $\mathcal{H}$ consists of all points $\lambda \in \mathbb{R}$ such that $H-\lambda I$ fails to be invertible. Using classical functional analytic techniques one can show that the operators, $\Delta_V$ and $\Delta_f$,  only have discrete spectrum. In other words, for $H=\Delta_V$ or $H=\Delta_f$, $\sigma(H)$ consists only of eigenvalues $\lambda$ such that there exists an element $u$ in the domain of the operator, $H$, satisfying
\[
H u = \lambda u.
\]
Moreover, these eigenvalues only accumulate at infinity.  An immediate consequence of the existence of a short-time asymptotic expansion is Weyl's law for the growth rate of its eigenvalues.
\begin{cor} \label{th-weyl}
Let $(M,g)$ be a compact, smooth, $n$-dimensional Riemannian manifold with Laplace operator, $\Delta$.  Let $V \in L^\infty(M,g)$.
Let $\{\lambda_k\}$ denote the eigenvalues of the Schr\"odinger operator, $\Delta_V = \Delta + V$ with $\lambda_1 \leq \lambda_2 \leq \ldots$.  For a function $f$ on $M$ with
$$\Delta f \textrm{ and } \nabla f \in L^\infty(M, g),$$
let $\{\mu_k\}$ denote the eigenvalues of the drifting Laplacian $\Delta_f = \Delta + \nabla f \cdot \nabla$ with $\mu_1 \leq \mu_2 \leq \ldots $.
Let
$$N_V(\Lambda) := \# \{ \lambda_k \leq \Lambda\}, \textrm{ and } N_f (\Lambda) := \# \{ \mu_k \leq \Lambda \}.$$
Then
$$\lim_{k \to \infty} \frac{N_V(\Lambda) (2\pi)^n}{\Lambda^{n/2} \omega_n} = \Vol (M) = \lim_{k \to \infty} \frac{N_f(\Lambda) (2\pi)^n}{\Lambda^{n/2} \omega_n},$$
where above $\Vol(M)$ denotes the volume of $M$ with respect to the Riemannian metric $g$, and $\omega_n$ denotes the volume of the unit ball in $\R^n$.
\end{cor}

As one readily sees from the above result, Weyl's law for the drifting Laplacian is in fact independent of the weight function $f$, whenever $\nabla f$ and $\Delta f$ are in $L^\infty$. In fact, Weyl's law only depends on the dimension of the manifold and its Riemannian volume. Although some experts provide (without proof)  the correct Weyl law  in the case $f\in C^2$ (see for example p. 640 of \cite{css}) it also often appears in literature (also without proof) with the weighted volume of the manifold. We anticipate that our paper will not only clarify this inconsistency, but also provide a more general setting in which the result is true.  We prove Corollary \ref{th-weyl} in \S \ref{sskw} using the the short time asymptotic behavior of the heat trace together with the Karamata Lemma.

The independence of Weyl's law from the weight function $f$ for the drifting Laplacian comes in contrast to estimates for its eigevalues, which {\it do } depend on the weight function.  For example in \cite{asma1}, A. Hassennezhad demonstrates upper bounds for $\Delta_V$ and hence $\Delta_f$, under the assumption that $V$ is continuous.
J.-Y. Wu and P. Wu demonstrate lower bounds for the eigenvalues of $\Delta_f$ under the assumption that the $f$-Ricci curvature is bounded below \cite{wuwu}.  Their estimates depend on lower bounds for the Bakry-\'Emery Ricci tensor, hence they also rely on $f$ being a $C^2$ function. As a result, one would also not anticipate that a Weyl's law would hold for such general $f$ as in Corollary \ref{th-weyl}, nor that it would be independent of $f$.
Our work complements the upper and lower eigenvalue estimates obtained by other authors and it may be useful to combine it with those results to obtain sharper eigenvalue estimates.

We will show in Section \S \ref{S3} that the remainder estimate in the asymptotic expansion given in Proposition \ref{prophttr} is determined by, and conversely it also distinguishes the regularity of the potential function.
\begin{thm}  \label{thm3.2} Assume $V \in L^\infty$.  Then the heat trace has an asymptotic expansion as $t \downarrow 0$ of the form
\begin{equation*}
\begin{split}
\tr e^{-t\Delta_V} & =  \frac{\Vol (M)}{(4\pi t)^{n/2}} +  \frac{t}{3(4\pi t)^{n/2}}  \int_M \left(\frac{1}{6} \, K(z)   -  V(z) \right)\, d\mu (z) \\
 &  + \frac{t^2}{(4\pi t)^{n/2}} \left[ a_{0,2}  + \frac{1}{2} \|V\|_{L^2}^2
 - \frac{1}{6}  \int_M K(z) V(z)  \,d\mu (z)    \right] + O(t^{3-n/2}),
\end{split}
\end{equation*}
if and only if $V \in H^1$.
\end{thm}

This regularity result was inspired by the work of H. Smith and M. Zworski for Schr\"odinger operators on $\R^n$ \cite{smith-z}.

\begin{cor} \label{cor-isosp}
Assume $V, \tilde{V}\in L^\infty$ and that the Schr\"odinger operators $\Delta_V, \Delta_{\tilde{V}}$ are isospectral. Then the potential, $V$, is in $H^1$ if and only if $\tilde{V}$ is in $H^1$.
\end{cor}

In the case that $f$ is smooth, we compute in Section \ref{S4} the full small time asymptotic expansion expansion of the heat trace associated to $\Delta_f$, $H_f(t,x,x)$.
\begin{thm} \label{thmhtr}
For a smooth weight function $f$, $H_f(t,x,x)$ has the short time asymptotic expansion
\begin{equation}  \label{htr1}
H_f(t,x,x) \sim (4 \pi t)^{-n/2} e^{f(x)} \,\sum_{i=0}^\infty u_i(x,x) \, t^i
\end{equation}
where the $u_i$ are defined by \eref{u_i}.

If $\{ \lambda_i\}$ is the spectrum of the drifting Laplacian $\Delta_f$, then
\begin{equation} \label{htr2}
\sum_k e^{\lambda_k t} \sim (4 \pi t)^{-n/2} \sum_{i=0}^\infty a_i \, t^i
\end{equation}
where
\[
a_i = \int_M u_i(x,x) \, d\mu(x).
\]
We emphasize that the integral on the right side is taken with respect to the Riemannian volume form $d\mu$.
\end{thm}

Finally, we give isospectrality results for the drifting Laplacian on compact manifolds.
\begin{thm} \label{isosp}
Consider two weighted  manifold $(M^n, g_M, d\mu_{f_1} )$ and $(N^m, g_N, d\mu_{f_2})$, with $f_i$ as in Corollary \ref{httr_drift} on the respective manifolds. Suppose that the drifting Laplacian $\Delta_{f_1}$ on $M$ and the drifting Laplacian $\Delta_{f_2}$ on $N$  are isospectral. Then the two manifolds must have the same dimension and the same Riemannian volume.

If in addition $M$ and $N$ are orientable surfaces, and the two weight functions have equal Dirichlet norms, $\int_M  \, |\n f_1 |^2  \, d\mu_M = \int_N  \, |\n f_2 |^2  \, d\mu_N,$ then they are diffeomorphic.

If $(M^n, g_M)$ is of dimension $n \leq 3$, then the set of isospectral drifting Laplacians, $\Delta_f$, with smooth $f$, is compact in $C^\infty (M)$.
\end{thm}

\section{Heat trace expansions}\label{s-heat}

\subsection{Heat trace for a Schr\"odinger operator with an $L^\infty$ potentials}
In this subsection will prove Proposition  \ref{prophttr}. Recall that $(M, g)$ be a smooth, compact, Riemannian manifold of dimension $n$ and $\Delta_V = \Delta + V$ is a Schr\"odinger operator with $V \in L^\infty (M)$.

\begin{proof}[Proof of Proposition \ref{prophttr}]
The estimates of \cite{smith-z} \S 3 p. 465, show that the semigroup $e^{-t\Delta_V}$ is trace class, but since their results are over $\R^n$, we recall the required calculations and estimates in the setting of a compact Riemannian manifold. Using Duhamel's principle, one can generate an expression for $e^{-t\Delta_V} - e^{-t\Delta}$ which is given by
$$e^{-t\Delta_V} - e^{-t\Delta} = \sum_{k \geq 1} W_k(t),$$
with
\begin{equation} \label{W_kdef}
W_k (t) = (-1)^k \int_{0<s_1 < \ldots < s_k < t} e^{-(t-s_k) \Delta} V e^{-(s_k - s_{k-1}) \Delta} V \ldots V e^{-s_1 \Delta} ds_1 \ldots ds_k.
\end{equation}
We first observe that we have the $L^2$ operator norm bound
$$\|W_k (t)\|_{L^2 \to L^2} \leq \|V\|_{\infty} ^k t^k/k!,$$
since the integrand is $L^2$ bounded by $\|V\|_\infty ^k$, and the volume of integration is $t^k/k!$.  Then, we also have a bound on the trace class norm,
\begin{equation} \label{W_ke1}
\|W_k(t)\|_{L^1} \leq C^k k^{n/2} t^{k-n/2} /k!.
\end{equation}
This follows immediately by the same arguments given on p. 365 of \cite{smith-z}. In fact, the proof can be simplified in our case: since we are working on a compact manifold, there is no need to use a cut-off function $\chi$, and we can obtain the estimate directly by applying the operator $e^{-s \Delta}$. 

Consequently, the operator $e^{-t\Delta_V} - e^{-t\Delta}$ is trace class, and the trace may be exchanged with summation to get
$$\tr(e^{-t\Delta_V} - e^{-t\Delta}) = \sum_{k \geq 1} \tr (W_k (t)).$$
We will compute the trace of the first term, $W_1$.  Let $H_0 (t,x,y)$ denote the heat kernel of the semigroup
$e^{-t\Delta}$. The Schwartz kernel of $W_1$ is given by
$$W_1 (t, x, y) = - \int_{0} ^t \int_M H_0 (s,x,z) H_0 (t-s, z, y) V(z) \;dz ds, $$
where to simplify notation, we have used $dz$ to indicate $d\mu(z)$. To compute its trace, we set $y=x$ and integrate,
\begin{equation} \label{eqn-trwk0}
\tr W_1 (t) = -\int_0 ^t \int_M \int_M H_0 (s,x,z) H_0 (t-s,z,x) V(z) \; dz dx ds
\end{equation}
Since we are on a compact manifold, the Laplacian has an $L^2$ orthonormal basis of eigenfunctions $\{ \phi_k\}$ with corresponding eigenvalues $\lambda_k$. Its heat kernel then has a simple expression with respect to this basis
\begin{equation} \label{H0_onb}
H_0 (s,x,z) = \sum_{k \geq 1} e^{-\lambda_k t} \phi_k (x) \phi_k (z).
\end{equation}
Expressing the heat kernels in \eref{eqn-trwk0} in terms of this orthonormal basis and using the orthonormality of the eigenfunctions, we compute the integral with respect to $x$ to arrive at
\begin{equation}  \label{W_1e1}
\tr W_1 (t) = - \int_0 ^t \int_M H_0(t, z, z) V(z) dz ds = -t \int_M H_0 (t,z,z) V(z) \;dz.
\end{equation}

Next, we use the fact that the heat kernel has a local expansion along the diagonal as $t\downarrow 0$
$$H_0(t,z,z) = \frac{1}{(4\pi t)^{n/2}} + \frac{t K(z)}{6(4\pi t)^{n/2}} + r(t) \quad \text{where}  \quad r(t) = O(t^{2-n/2}),$$
and $K(z)$ is the scalar curvature of $(M,g)$ (see for example \cite{Ro}*{Chapter 3}).
Since $V \in L^\infty$ and $M$ is compact, it then follows that as  $t \downarrow 0$
\begin{equation} \label{trw1}
\tr W_1 (t) = -t (4\pi t)^{-n/2} \left[ \int_M V(z) dz + \frac{t}{6} \int_M K(z) V(z) dz\right] +  O(t^{3-n/2}).
\end{equation}

In fact, we can also compute the full asymptotic expansion of the term $\tr W_1 (t)$ as $t \downarrow 0$. Recalling the identity \eref{W_1e1} for $W_1$ and applying the local heat trace expansion of $H_0$ together with the assumption that $V \in L^\infty$, we obtain
\begin{equation*}
\tr W_1 (t) \sim -t (4\pi t)^{-n/2} \sum_{j \geq 0} t^j \int_M u_{0,j} (z,z) V(z) dz.
\end{equation*}
The functions $u_{0,j}(z,z)$ are the local heat invariants of the heat kernel $H_0$ on the Riemannian manifold $(M, g)$. Consequently, the $u_{0,j}$ are {\it  independent of }  $V$.  It is in general very difficult to express these functions in terms of geometric and topological invariants of the manifold. We refer the interested reader to Rosenberg's results in \cite{Ro}, and perhaps more conveniently to equation \eref{u0_i} for an integral expression of the $u_{0,j}$. In the case $j=1$, as we have mentioned above, $u_{0,1}(z,z)=K(z)/6$ where $K(z)$ is the scalar curvature of the manifold  \cite{Ro}. We would also like to remark that in the expansion of the trace of $W_1$ only $V$ appears and none of its higher order derivatives.

To complete the proof, we continue to analyze the next term in the expansion.
By definition \eref{W_kdef}, the trace of $W_2$ is
$$\tr W_2 (t) = \int_{M^3} \int_{0<r<s<t} H_0(t-s,x,y) H_0(s-r,y,z) H_0(r,z,x) V(y) V(z) \;dr ds dx dy dz.$$
We again use the orthonormal basis expansion \eref{H0_onb} to express $H_0(t-s,x,y)$ and $H_0(r,z,x)$, and integrate with respect to $x$ to obtain
$$\tr W_2 (t) = \int_{M^2} \int_{0<r<s<t} H_0(t-s+r,y,z) H_0(s-r,z,y) V(y) V(z) \; dr ds dy dz.$$

We can now apply a clever time-substitution technique to further simplify this expression (see for example \cite{smith-z}*{pp. 467--468}).  Letting $u=t-s$ and suppressing the integrals over space we get
$$ \int_{r+u<t, \; 0<r,u} H_0(u+r,y, z) H_0(t-u-r,y, z) V(y) V(z) \; dr du.$$
Setting $r=tv-u$, and observing that $dr du = t dv du$,  we have
\begin{equation*}
\begin{split}
\int_{v=0}^1 & \int_{0<u<tv}  H_0 (tv, y, z) H_0 (t(1-v), y, z) V(y) V(z) t \; du dv\\
&= t^2 \int_{v=0}^1 H_0 (tv, y, z) H_0 (t(1-v), y, z) V(y) V(z) v \; dv,
\end{split}
\end{equation*}
where for the last integration we have used that the integrand is independent of $u$. Moreover, the symmetry of the last integral with respect to the map $v \mapsto 1-v$, further simplifies it to
\begin{equation*}
\frac{t^2}{2} \int_0 ^1 H_0 (tv, y, z) H_0 (t(1-v), y, z) V(y) V(z) \;dv.
\end{equation*}

Let $\cN_\eps$ be a small $\eps$-neighborhood of the diagonal,  in $M \times M$.
By the rapid off-diagonal decay and the short time local asymptotic expansion for the two heat kernels $H_0 (tv, y, z)$ and $H_0 (t(1-v), y, z)$ (see \S 3 of \cite{Ro}),
\begin{equation*}
\begin{split}
\tr W_2 (t) = \frac{t^2}{2} \, (4\pi t)^{-n} \int_{\cN_\epsilon} \int_{v=0} ^1 &(1-v)^{-n/2} v^{-n/2} e^{-d(y,z)^2/(4tv(1-v))} \\
& \cdot u_0 (y, z)^2 V(y) V(z) \;dv dy dz  + O(t^{3-n/2}) + O_\eps(t^\infty)
\end{split}
\end{equation*}
where $O_\eps(t^\infty)$ depends on $\eps$.

We now compare the short time asymptotic expansion of the product
\begin{equation*}
\begin{split}
H_0 &(tv, y, z) H_0 (t - tv, y, z) \\
&= \left( (4\pi t)^{-n/2} \right)^2 (v(1-v))^{-n/2} e^{-d(y,z)^2/(4tv(1-v))} u_0(y,z)^2 + O(t^{1-n/2}),
\end{split}
\end{equation*}
to
$$H_0(tv(1-v), y, z) =  (4\pi t)^{-n/2} (v(1-v))^{-n/2} e^{-d(y,z)^2/(4tv(1-v))} u_0 (y,z) + O(t^{1-n/2}).$$
Due to the $O(t^\infty)$ off-diagonal decay it suffices to compare $u_0 (y,z)$ and $u_0 (y,z)^2$ for $y \approx z$, recalling that $u_0(z,z)=1$. Since $u_0$ is smooth and $M$ is compact, for any $1>\eps > 0$, there exists $\delta > 0$ such that
$$d(y,z) < \delta \implies |u_0 (y,z) - u_0 (z,z)| = |u_0 (y,z) - 1| < \frac{\eps}{2} \implies |u_0(y,z)| < 1+ \frac{\eps }{2},$$
and therefore
\begin{equation} \label{u0est}
|u_0 (y,z)^2 - u_0 (y,z)| = |u_0 (y,z)| | u_0 (y,z) - 1|  < \left( 1+ \frac{\eps}{2} \right) \frac{ \eps  }{2} < \eps.
\end{equation}

By the heat kernel expansion,
\begin{equation*}
\begin{split}
|H_0 &(tv, y, z) H_0 (t - tv, y, z) - (4\pi t ) ^{-n/2} H_0 (tv(1-v),y,z) |\\
&= (4\pi t)^{-n} (v(1-v))^{-n/2} e^{-d(y,z)^2/(4tv(1-v))} |u_0(y,z)|\; |u_0 (y,z) - 1| +O(t^{1-n/2}).
\end{split}
\end{equation*}
We also know that for $y \neq z$, the right side is rapidly decaying as $t \downarrow 0$.  Moreover, for any $t>0$, choosing a small enough $\eps$-neighborhood $\cN_\eps$ (depending on $t$) we can make the right side above as small as desired by \eref{u0est}.  Consequently, we get the appoximation
\begin{equation} \label{W_2e5}
\begin{split}
\tr W_2 (t) = \frac{t^2}{2} (4\pi t)^{-n/2} \int_0 ^1 \int_{M^2} & H_0 (tv(1-v), y, z) V(y) V(z) \;dy dz dv \\
&+ O(t^{3-n/2})
\end{split}
\end{equation}

Note that
$$ \int_{y \in M} H_0 (tv(1-v), y, z) V(y) \;dy = U_V (tv(1-v), z),$$
is the solution to the heat equation with initial data $V$ at time $tv(1-v)$.  By \eref{W_2e5} we have
\begin{equation} \label{W_2e6}
\left|  2 (4\pi t)^{n/2} t^{-2} \tr W_2 (t) - \int_0 ^1 \int_M  U_V (tv(1-v), z) V(z) \; dz dv \right| = O(t).
\end{equation}

We recall that the solution to the heat equation with initial data $V$ at time $tv(1-v)$
converges to $V$ as $t\downarrow 0$, uniformly in $L^2 (M)$ for all $v \in (0,1)$. In other words,
$$\| U_V (tv(1-v), z) - V(z)\|_{L^2} \to 0 \textrm{ uniformly as $t\downarrow 0$ for all $v\in (0,1)$.}$$
and as a result,
\begin{equation*}
\begin{split}
&\left| \int_0 ^1  \int_M U_V (tv(1-v), z) V(z) dv dz - \int_{0} ^1 \int_{M}   |V(z)|^2 \;dz \; vdv \right|  \\
& \qquad \qquad \leq \left( \|V\|_{L^2(M)} \sup_{v \in [0,1]} \| U_V (tv(1-v)) - V\|_{L^2 (M)}  \right) \to 0 \textrm{ uniformly as } t \downarrow 0.
\end{split}
\end{equation*}
However, we do not know a priori the {\it  rate} at which the right side converges as $t \downarrow 0$. Combining this estimate with \eref{W_2e6} we therefore obtain
\begin{equation} \label{W_2e1}
\tr W_2 (t) = \frac{t^2}{2}  (4\pi t)^{-n/2} \left[ \|V\|_{L^2} ^2  + o(1) \right].
\end{equation}

To complete the proof, we recall the bound in  \eref{W_ke1}, which shows that
\begin{equation} \label{eqn-trwk} \tr \sum_{k\geq 3} W_k (t) \leq C t^{3-n/2}, \quad 0<t\leq 1. \end{equation}
As a result, under the very general assumption that $V \in L^\infty$, we have a precise expression for the heat trace up to order $t^{2-n/2}$ and a remainder estimate,
\begin{equation*}
\begin{split}
\tr e^{-t \Delta_V} = & \frac{\Vol (M)}{(4\pi t)^{n/2}} + \frac{t}{(4\pi t)^{n/2}} \int_M \left( \frac 13 \, K(z)   - V(z)  \right) \;d\mu (z) \\
&   + \frac{t^2}{(4\pi t)^{n/2}}  \left[ a_{0,2} + \frac{1}{2} \|V\|_{L^2} ^2 - \frac{1}{6} \int_M  K(z) V(z)  \;d\mu (z)  \right] + o(t^{2-n/2}),
\end{split}
\end{equation*}
as $t \downarrow 0.$
\end{proof}

We note that the  first, second, and fourth terms come from the corresponding  terms in the short time heat trace expansion of $e^{-t\Delta}$ on $M$.  Moreover, the potential $V$ does not appear in the leading order asymptotics; it first makes an appearance in the coefficient of  $t^{1-n/2}$.  This is the crux of the matter and the reason why $V$ does not affect Weyl's law for a Schr\"odinger operator.  However, based on the work of Hassanezhad \cite{asma1} and Wu and Wu \cite{wuwu}, one {\it does } expect that more refined estimates of the eigenvalues of a Schr\"odinger operator may indeed depend on the potential.

\subsection{Weyl's law via the heat trace and Karamata Lemma}  \label{sskw}

In general, there are two main techniques which can be used to prove Weyl's law.  The first technique is known as {\it  Dirichlet-Neumann bracketing},  and this is one used in Weyl's classical proof \cite{weyl1}.  The second technique uses an element of the functional calculus of the Laplacian, such as the resolvent, wave group, or heat semi-group, together with a suitable Tauberian theorem.  Here, we use the short time asymptotic behavior of the heat trace together with Karamata's Tauberian Lemma to prove Corollary \ref
{th-weyl}.

\begin{lem}[Karamata] \label{klemma} Assume $d\nu$ is a non-negative measure on $(0, \infty)$, and $\alpha \geq 1$.  Assume that
$$\int_0 ^\infty e^{-t\lambda} d\nu(\lambda) < \infty \quad \forall t > 0$$
and
$$\lim_{t \downarrow 0} t^\alpha \int_0 ^\infty e^{-t\lambda} d\nu(\lambda) = c \in (0, \infty).$$
Then for all continuous functions $g$ on $[0,1]$ we have
$$\lim_{t \downarrow 0} t^\alpha \int_0 ^\infty g(e^{-t\lambda}) e^{-t\lambda} d\nu (\lambda) = \frac{c}{\Gamma(\alpha)} \int_0 ^\infty g(e^{-t}) t^{\alpha - 1} e^{-t} dt.$$
\end{lem}

The relationship between the short time asymptotic behavior of the heat kernel together with the large $\Lambda$ asymptotic behavior of the eigenvalue counting function via Karamata's Lemma is not completely obvious.  To  explicate  this, we recall that a natural way to relate the heat trace to the counting function is by introducing the spectral measure, $\nu$, with
$$d\nu(\lambda) = \sum_{k \geq 1} \delta(\lambda - \lambda_k).$$
With respect to this measure, the heat trace can be written as
$$\int_0 ^\infty e^{-t\lambda} d\nu(\lambda).$$
From this perspective, the existence of a one-term asymptotic expansion together with a remainder estimate for the heat trace as $t\downarrow 0$ is precisely the statement that there exists $c \in (0, \infty)$, and a constant $\alpha$  with
\begin{equation} \label{klc} \lim_{t \downarrow 0}t^\alpha \int_0 ^\infty e^{-t\lambda} d\nu (\lambda) = c. \end{equation}

Next, one would like to connect the above assumption to the counting function asymptotics as the spectral parameter tends to infinity.  To do this, consider a function, $g$, which is equal to $\frac{1}{x}$ on some compact subset of $[0, \infty)$ and suitably cut-off.  Then, one has
$$t^\alpha \int_0 ^\infty g(e^{-t\lambda}) e^{-t\lambda} d\nu(\lambda) =t^\alpha  \int_{\supp(g(e^{-t\lambda}))} d\nu(\lambda) = t^{\alpha} \# \{ \lambda_k \textrm{ with } \lambda_k \in \supp(g(e^{-t\lambda})) \}.$$
This calculation gives some hope that perhaps with a clever choice of the function $g$, one can use (\ref{klc}) and the  preceding equation to demonstrate Weyl's law.  It is then natural to investigate the existence of
$$\lim_{t \downarrow 0} t^\alpha \int_0 ^\infty g(e^{-t\lambda}) e^{-t\lambda} d\nu(\lambda),$$
for a suitably general class of functions $g$.  We hope that these insights demystify Karamata's seemingly magical result.

We shall apply the Karamata Lemma to a function which does not in fact satisfy the hypotheses of the Lemma.  For the sake of completeness, we shall therefore prove a slight generalization of Karamata's Lemma.

\begin{lem} \label{klemma2}
Let $g$ be a bounded, non-negative, piecewise continuous function on $[0,1]$.  Assume that $\nu$ is a non-negative measure, $\alpha \geq 1$, and
$$\int_0 ^\infty e^{-t \lambda} d\nu (\lambda) < \infty \, \forall t > 0, \quad \lim_{t \downarrow 0} t^\alpha \int_0 ^\infty e^{-t\lambda} d\nu(\lambda) = c \in (0, \infty).$$
Then, we have
$$\lim_{t \downarrow 0} t^\alpha \int_0 ^\infty g(e^{-t \lambda}) e^{-t \lambda} d\nu(\lambda) = \frac{c}{\Gamma(\alpha)} \int_0 ^\infty g(e^{-t}) e^{-t} t^{\alpha - 1} dt.$$
\end{lem}
\begin{proof}
We begin by introducing the notations
\begin{equation}
F_t(g) = t^\alpha \int_0^\infty g\left(e^{-t\lambda}\right)e^{-t\lambda} \mathrm d\nu(\lambda),\quad G(g) = \frac{c}{\Gamma(\alpha)}\int_0^\infty g\left(e^{-t}\right)t^{\alpha-1} e^{-t}\mathrm dt.
\end{equation}

Let $\{g_n\}$ be a sequence of continuous functions such that
$$0 \leq g_1(x) \leq g_2(x) \leq \ldots \leq g_n (x) \leq g_{n+1} (x) \leq \ldots \leq g(x) \leq M, \quad \forall x \in [0,1],$$
and
$$\lim_{n \to \infty} g_n (x) = g(x) \quad \forall x \in [0,1].$$
Because $g$ is non-negative and bounded, we have the estimate
$$0 \leq G(g) \leq \frac{c}{\Gamma(\alpha)} \int_0 ^\infty M t^{\alpha-1} e^{-t} dt = cM.$$
This will allow us to apply dominated convergence arguments.

Since each $g_n$ is continuous,
$$\lim_{t \downarrow 0} F_t (g_n) = G(g_n).$$
Moreover, under these assumptions,
$$\lim_{n \to \infty} G(g_n) = G(g), \quad \lim_{n \to \infty} F_t (g_n) = F_t (g) \quad  \forall t > 0.$$
Therefore,
$$G(g_n) = \lim \inf_{t \downarrow 0} F_t (g_n) \leq \lim \inf_{t\downarrow 0} F_t (g), \quad \forall n,$$
so letting $n \to \infty$, we have
$$G(g) \leq \lim \inf_{t \downarrow 0} F_t (g).$$
On the other hand,
\begin{equation*}
\begin{split}
\lim \sup_{t \downarrow 0} F_t (g) &= \lim \sup_{t \downarrow 0} F_t (\lim_{n \to \infty} g_n) \leq \lim \sup_{t \downarrow 0} \lim \sup_{n \to \infty} F_t (g_n) \\
&\leq \lim_{t \downarrow 0, n \uparrow \infty} \sup_{s \leq t, m \leq n} F_s (g_m)= \lim_{t \downarrow 0, n \uparrow \infty} \sup_{s \leq t} F_s (g_n) \leq \lim \sup_{n \to \infty} G(g_n) = G(g).
\end{split}
\end{equation*}
We therefore obtain
$$\lim_{t \downarrow 0} F_t (g) = G(g).$$
\end{proof}


\begin{proof}[Proof of Corollary \ref{th-weyl}]
As usual, we assume that $(M,g)$ is a compact Riemannian manifold, and $V \in L^\infty(M,g)$.
Let $\{ \lambda_k \}_{k \geq 1}$ be the spectrum of the Schr\"odinger operator, $\Delta + V$, on $M$.   We shall apply the Karamata Lemma with the measure
$$d\nu := \sum_{k \geq 1} \delta_{\lambda_k}.$$
Then we note that
$$\int_0 ^x d\nu(\lambda) = N(x) = \# \{ \lambda_k \leq x \} - n_-.$$
Above, $n_-$ is the number of negative eigenvalues of the Schr\"odinger operator.  It is well known that $n_-$ is finite whenever the potential $V$ is bounded below. For the sake of completeness, however, we will provide the brief argument for this fact here. Let $\phi_k$ be the unitary eigenfunction corresponding to $\lambda_k$ such that $\Delta_V \phi_k = \lambda_k \phi_k$ and $\int_M \phi_k ^2 =1$. First we note that 
\[
\lambda_k = \int_M |\nabla \phi_k|^2 + \int_M V \phi_k^2. 
\]
Since
$$\left| \int_M V \phi_k^2 \right| \leq \|V\|_\infty,$$
we have
$$\lambda_k \geq \int_M |\nabla \phi_k|^2 - \|V\|_\infty.$$
The second term above is fixed, and the first term is positive for all non-constant eigenfunctions.  The spectral theorem in this setting implies that the eigenvalues may only accumulate at $\pm \infty$.   By the above estimate, the eigenvalues may only accumulate at $+ \infty$, and there are at most finitely many negative eigenvalues.

We shall apply Lemma \ref{klemma2} to the function,
\begin{equation} \label{cleverf}
g(x) = \left\{ \begin{array}{ll} 0;&\quad x\in\left[0,e^{-1}\right]\cup [1,\infty) \\ \frac{1}{x};&\quad x\in \left(e^{-1},1\right)\end{array} \right..
\end{equation}

Observe that $e^{-t\lambda} \in \left[0,e^{-1}\right]$, $\Leftrightarrow$ $t\lambda \geq 1$ $\Leftrightarrow$ $t\geq \frac q \lambda$ $\Leftrightarrow$ $\lambda \geq \frac 1 t$. On the other hand $e^{-t\lambda} \geq 1$ $\Leftrightarrow$ $t\lambda \leq 0$. Therefore, for $t >0$,
\begin{equation*}  
g\left(e^{-t\lambda}\right) = \left\{ \begin{array}{ll} 0; &\quad \lambda \geq \frac 1 t \;\mathrm{or}\;\lambda < 0 \\ e^{t\lambda}; &\quad 0 <\lambda < \frac 1 t\end{array}\right..
\end{equation*}

By Proposition \ref{prophttr},
$$\lim_{t \downarrow 0} t^{n/2} \int_0 ^\infty e^{-t \lambda} d\nu(\lambda) = \lim_{t\downarrow 0} t^{n/2}\left( \tr e^{-t\Delta_V} - \sum_{\lambda_k < 0} e^{-\lambda_k t} \right)= \frac{ \Vol(M)}{(4\pi)^{n/2}}.$$
This follows from the fact that the sum over the negative eigenvalues is a finite sum and therefore the factor of $t^{n/2}$ kills it as $t \downarrow 0$.

We therefore let
$$c :=  \frac{ \Vol(M)}{(4\pi)^{n/2}}, \quad \alpha = \frac{n}{2}$$
Applying Karamata's Lemma with the aforementioned $g$, $c$, and $\alpha$,
\begin{equation*}
\lim_{t\downarrow 0} t^\alpha \int_0^\infty g\left(e^{-t\lambda}\right) e^{-t\lambda} \mathrm d\nu(\lambda) = \frac{c}{\Gamma\left(\frac n 2\right)} \int_0^\infty g\left(e^{-t}\right) t^{\frac n 2 -1} e^{-t} \mathrm dt.
\end{equation*}

By the computation
\begin{equation*} \label{form_f}
g\left(e^{-t }\right) = \left\{ \begin{array}{ll} 0; &\quad t \geq 1 \;\mathrm{or}\; t < 0 \\ e^{t }; &\quad 0 < t <  1 \end{array}\right..
\end{equation*}

we have
\begin{equation*}
\lim_{t \downarrow 0} t^\alpha \int_0^{1/t} e^{t\lambda} e^{-t\lambda} \mathrm d\nu(\lambda) = \lim_{t \downarrow 0} t^\alpha N\left(\frac 1 t\right) = \frac{c}{\Gamma\left(\frac n 2\right)}\int_0^1 e^t t^{n/2-1} e^{-t} \mathrm dt = \frac{c}{\frac{n}{2}\Gamma\left(\frac{n}{2}\right)}.
\end{equation*}
Above we have assumed $n\geq 2$. Thus $t^{n/2} N\left(\frac{1}{t}\right) \to \frac{2c}{n\Gamma(n/2)}$ which is equivalent to
\begin{equation*}
N(\lambda) \sim \frac{2\lambda^{n/2}}{n\Gamma\left(\frac{n}{2}\right)}\frac{\Vol(M)}{(4\pi)^{n/2}}, \quad \lambda \to \infty.
\end{equation*}
 Since
 $$\omega_n = \frac{2 \pi^{n/2}}{n \Gamma\left( \frac n 2 \right)},$$
 the proof of Weyl's law for the operator $\Delta + V$ is complete.

To prove Weyl's law for the drifting Laplacian on a weighted manifold, we use the fact that is is unitarily equivalent to a Schr\"odinger operator. We consider a function $f$ on the manifold, $M$, such that $\Delta f$ and $\nabla f$ are both in $L^\infty (M,g)$.  Let $L^2_f$ denote the set of $L^2$ integrable functions on $M$ with respect to the weighted measure $d\mu_f$,  $L_f^2(M) =\{ u\ \big| \int_M u^2 e^{-f} \, d\mu <\infty \} $. Using the transformation $T: L^2 \to L^2_f$ given by $T(u)=e^{\frac 12 \, f}u$ we have that the drifting Laplacian is unitarily equivalent to the Schr\"odinger operator,
$$\Delta + V, \quad \text{where}  \quad V = \frac{1}{2} \Delta f + \frac{1}{4} | \nabla f|^2.$$
Under these assumptions the potential function $V = \frac{1}{2} \Delta f + \frac{1}{4} | \nabla f|^2 \in L^\infty (M,g)$. Consequently, although these operators need not have the same eigenfunctions, they do have the same eigenvalues.  In other words, for the above $V$ and $f$, the operators $\Delta_V$ and $\Delta_f$ are isospectral.   We therefore apply the preceding proof of Weyl's law for the operator $\Delta + V$ and conclude that the operator $\Delta_f$ obeys the same Weyl law as $\Delta + V$.
\end{proof}

\begin{remark}
It is straightforward to prove that Weyl's law is equivalent to
$$\lim_{k \to \infty} \frac{\lambda_k}{k^{2/n}} = \frac{ \omega_n^{2/n}}{(2\pi)^{2} \Vol (M)^{2/n}}.$$
Equivalently, one has
$$\lambda_k = k^{2/n} \frac{ \omega_n^{2/n}}{(2\pi)^{2} \Vol  (M)^{2/n}} + o(k^{2/n}), \quad k \to \infty.$$
\end{remark}


\section{Heat trace and regularity of potential} \label{S3}
In this section we demonstrate that the sharpness of the remainder estimate in the short time asymptotic expansion of the heat trace is equivalent to higher regularity for the potential. It turns out that the regularity result of Theorem \ref{thm3.2} is determined entirely by the term $\tr W_2$ we saw in the previous section.

\begin{prop} \label{prop_reg} Assume that $V \in L^\infty$.  If $\tr W_2(t) $ has an expansion as $t \downarrow 0$ of the form
$$\tr W_2 (t) = \frac{t^2}{2 (4\pi t)^{n/2}}\left( c + O(t) \right).$$
then $V \in H^1$, and
$$c = \|V\|_{L^2} ^2.$$
Conversely, if $V \in H^1$, then $\tr W_2 (t) $ has such an expansion and
$$\tr W_2 (t) = \frac{t^2}{2 (4\pi t)^{n/2}}\left( \|V\|_{L^2} ^2 + O(t) \right).$$
\end{prop}
\begin{proof}

To prove the proposition, we begin by  assuming that
\[
\tr W_2 (t)  = \frac 12 \,  t^2 \, (4\pi t)^{-n/2}  \left[ c +  O(t)\right].
\]

Since $V\in L^\infty$,   (\ref{W_2e1}) gives $c = \|V\|_{L^2} ^2,$ and we now have
\begin{equation} \label{w2e3}
\tr W_2 (t)  = \frac 12 \,  t^2 \, (4\pi t)^{-n/2}  \left[ \|V\|_{L^2} ^2  +  O(t)\right].
\end{equation}

Comparing (\ref
{W_2e5}) to (\ref{w2e3}) we therefore obtain
\[
\frac {1}{t} \left|  \| V \|_{L^2}^2  - \int_0^1 \int_{M \times M}  H_0 (tv(1-v), y,z) \times  \, V(y) V(z)  dy dz dv \right| = \frac{O(t)}{t} \leq C
\]
for $t$ small enough.

By definition,
\[
\int_{M \times M}  H_0 (tv(1-v), y,z) \times  \, V(y) V(z)  dy dz = ( e^{-tv(1-v) \Delta_0}  V, V)
\]
therefore the above estimate also shows that
\begin{equation} \label{W_2e2}
\frac {1}{t} \left| \| V \|_{L^2}^2 -\int_0^1( e^{-tv(1-v) \Delta_0}  V, V) \,dv  \right|  \leq C  \textrm{ as } t \to 0^+.
\end{equation}

We make the following observation. For any $\lambda_k \geq 0$
\begin{align*}
\lim_{t\to 0^+ } \frac 1t \,  \left[ 1- \int_0^1 e^{-tv(1-v) \lambda_k} \,dv   \right] & =   \int_0^1 v(1-v) \lambda_k \,dv =  \frac 16 \, \lambda_k \\
& =  \frac 16 \,  \lim_{t\to 0^+ } \frac 1t \,  \left[ 1- \int_0^1 e^{-t\lambda_k} \,dv   \right].
\end{align*}

Using the orthonormal basis expansion \eref{H0_onb} for the for the heat kernel, we get
\begin{align*}(e^{-tv(1-v) \Delta_0} V, V) &= \sum_{k \geq 1} e^{-\lambda_k t v (1-v)} \int_{M \times M} \phi_k (z') V(z') \phi_k (z) V(z) dz dz'  \\
&= \sum_{k \geq 1} e^{-\lambda_k t v (1-v)}  |\widehat{V_k}|^2,
\end{align*}
where
$$\widehat{V_k} = \int_M V(z) \phi_k (z) dz,$$
is the $k^{th}$ Fourier coefficient of $V$ with respect to the basis.  We also note that
$$\|V\|_{L^2 (M)} ^2 = \sum_{k \geq 1} |\widehat{V_k}|^2.$$
Consequently,
$$\int_0 ^1 (e^{-tv(1-v) \Delta_0} V, V) dv - \|V\|^2 _{L^2} = \int_0 ^1 \sum_{k \geq 1} (e^{-\lambda_k t v (1-v)} -1) |\widehat{V_k}|^2 dv.$$
Since $V \in L^2$, the above expression converges absolutely and uniformly.  We may therefore compute
\begin{equation} \label{W_2e7}
\begin{split}
\lim_{t \to 0^+} \frac{1}{t}  &\left[ \int_0^1( e^{-tv(1-v) \Delta_0}  V, V) \,dv  - \| V \|_{L^2}^2 \right] = \lim_{t \to 0^+} \frac{1}{t}  \int_0 ^1 \sum_{k \geq 1} (e^{-\lambda_k t v (1-v)} -1) |\widehat{V_k}|^2 dv \\
&=\sum_{k \geq 1} \lim_{t \to 0^+} \frac{1}{t} \int_0 ^1 (e^{-\lambda_k t v (1-v)} -1) |\widehat{V_k}|^2 dv  = \frac{1}{6} \sum_{k \geq 1} \lim_{t \to 0^+} \frac{1}{t} \int_0 ^1 (e^{-\lambda_k t } -1) |\widehat{V_k}|^2 dv.
\end{split}
\end{equation}

On the other hand, by the definition of the heat operator,
$$\sum_{k \geq 1} e^{-\lambda_k t} |\widehat{V_k}|^2 = (e^{-t \Delta_0} V, V).$$
Substituting this in the right side of \eref{W_2e7}, we obtain
\begin{equation} \label{W_2e3}
\lim_{t\to 0^+ } \frac {1}{t} \left[ \int_0^1( e^{-tv(1-v) \Delta_0}  V, V) \,dv  - \| V \|_{L^2}^2 \right] = \frac 16 \,  \lim_{t\to 0^+ } \frac {1}{t} \left[ \int_0^1( e^{-t\Delta_0}  V, V) \,dv  - \| V \|_{L^2}^2 \right]
\end{equation}

The definition of the heat operator also gives that
\begin{equation} \label{W_2e4}
\begin{split}
\lim_{t\to 0^+ } \frac {1}{t} \left[ \int_0^1( e^{-t\Delta_0}  V, V) \,dv  - \| V \|_{L^2}^2 \right]&  =  - \lim_{t\to 0^+ } \frac{d}{dt}(e^{-t\Delta_0} V, V)   \\
&=  \lim_{t\to 0^+ }  (\Delta_0  e^{-t\Delta_0} V, V) =  \| \n V\|_{L^2}^2.
\end{split}
\end{equation}

The estimate in \eref{W_2e2} together with the identities \eref{W_2e3} and \eref{W_2e4} demonstrate that in fact $\| \n V\|_{L^2}^2 < \infty$  and we can conclude that $V$ is in $H_1$.  We have therefore proven that whenever $\tr W_2 (t)$ has an expansion as above, this implies that $V \in H_1$.

Conversely, if we assume that $V \in H_1$, then by \eref{W_2e4}, we see that we obtain the improvement in the remainder estimate in $\tr W_2(t)$, so that
$$\tr W_2 (t) = \frac{t^2}{2 (4\pi t)^{n/2}}\left( \|V\|_{L^2} ^2 + O(t) \right).$$
Hence, we see that $\tr W_2 (t)$ has an asymptotic expansion of this type if and only if $V \in H_1$.
\end{proof}

Finally, we use the proposition to prove Theorem \ref
{thm3.2}.
\begin{proof}
The \`a priori estimates on $\tr W_k (t)$ for $k \geq 3$ given in \eref{W_ke1} show that
$$\tr \sum_{k \geq 3} W_k (t) \leq C t^{3-n/2}, \quad 0 < t \leq 1$$
and in consequence
\[
\tr e^{-t\Delta_V} - \tr e^{-t\Delta} = \tr W_1 (t)  +   \tr W_2(t) + O(t^{3-n/2}).
\]

The asymptotic expansion for $\tr W_1(t)$ from \eref{trw1} together with the asymptotic expansion for the trace of the semigroup $e^{-t\Delta}$ give
\begin{equation*}
\begin{split}
\tr e^{-t\Delta_V} & -(4\pi t)^{-n/2} \left[   \Vol (M)  +   t \int_M   \left( \frac{1}{6} \,K(z)  - V(z) \right)\, d\mu (z) \right.\\ & \qquad \left. + t^{2} \left( a_{0,2}  - \frac{1}{6} \int_M   K(z) V(z) \, d\mu (z) \right) \right]   =  \tr W_2(t) + O(t^{3-n/2}).
\end{split}
\end{equation*}
The theorem therefore follows from Proposition \ref{prop_reg}.
\end{proof}

We proved in Proposition \ref{prophttr} that for $V\in L^\infty$ we have  an  expansion for the heat trace of $\Delta+V$  with an error term of order $ o(t^{2-n/2})$  for $t$ small. Theorem \ref{thm3.2} implies that if the error term is slightly better, then the potential must in fact belong to $H^1$. The converse is also true. This regularity result was inspired by the work of H. Smith and M. Zworski for Schr\"odinger operators on $\R^n$ \cite{smith-z}. There the authors were able to use the Fourier transform on Euclidean space to express the trace of the heat kernel and relate it to the $H^m$ norms of the potential, but also to obtain estimates for the terms $W_k$.  In this article we provide a generalization of some of their results to the case of compact manifolds. In our case, we are not able to use the nice properties of the Fourier transform, but instead  rely on the definition of the $H^1$ norm with respect to the heat operator, as well the expansion of the heat kernel with respect to its orthonormal basis of $L^2$ eigenfunctions. We believe that similar results should also hold on non-compact manifolds where the drifting Laplacian has a discrete spectrum and its heat operator can be expressed using such an orthonormal basis of eigenfunctions. We were recently informed by H. Smith that he has also independently worked out a generalization of his results with M. Zworski for complete manifolds.

\begin{cor} \label{httr_drift} Let $(M, g)$ be a smooth, compact Riemannian manifold with Laplace operator $\Delta$, and let $f$ be a real-valued function on $M$, such that
$$\Delta f \textrm{ and }  \nabla f \in L^\infty (M, g).$$
Then, the heat kernel for drifting Laplacian $\Delta_f = \Delta + \nabla f \cdot \nabla$ has a small-time asymptotic expansion as $t \downarrow 0$ of the form
\begin{equation*}
\begin{split}
\tr e^{-t\Delta_V}  =  & \frac{\Vol (M)}{(4\pi t)^{n/2}} + (4\pi t)^{-n/2} \left[ t \int_M \left(\frac{1}{6} \,   K(z)   -  \frac{1}{4} \, |\n f(z)|^2 \right) \; d\mu(z)   \right. \\
&\quad \quad + \left. t^2  \left( a_{0,2} - \int_M \frac{1}{6} \,   K(z) V(z)  \; d\mu(z) + \frac{1}{2} \|V\|_{L^2} ^2 \right) \right]+ R(t)
\end{split}
\end{equation*}
where the remainder $R(t) = o(t^{2-n/2})$ as $t \downarrow 0$.
Moreover, the remainder $R(t) = O(t^{3-n/2})$ if and only if $V \in H^1$.
\end{cor}

\begin{proof}
The corollary follows immediately from the observation that the Schr\"odinger operator $\Delta + V$ and the drifting Laplacian $\Delta + \nabla f \cdot \nabla$ are unitarily equivalent, and that $\int_M V(z)  \; d\mu(z) =  1/4  \int_M | \nabla f(z)|^2 \;d\mu(z)$.
\end{proof}

The proof of Corollary \ref{cor-isosp} is a simple consequence.
\begin{proof}[Proof of Corollary \ref{cor-isosp}]
If the Schr\"odinger operators $\Delta + V$ and $\Delta + \tilde{V}$ are isospectral, then they have the same heat trace.  By the preceding corollary, the remainders are either both $O(t^{3-n/2})$ or $o(t^{2-n/2})$, and therefore either both $V$ and $\tilde V$ are in $H^1$, or they are both not in $H^1$.
\end{proof}

\section{Heat trace for the drifting Laplacian with smooth weight function} \label{S4}
In this section we will give the classical method for obtaining the heat trace of the drifting Laplacian.   This is done via the parametrix method, as in the case of the Laplacian, and works particularly well for a smooth weight function $f$. Since we are mainly interested in the coefficients for the short time asymptotic expansion of the heat trace we will omit some of the simply technical elements of the arguments as they are identical to the unweighted case (we will refer the interested reader to \cite{Ro} for the details).

As the manifold is compact, there exists a uniform constant $\eps>0$ such that for any $x\in M$ the exponential map at $x$ is a diffeomorphism from the ball of radius $\eps$ in the tangent space onto $B_x(\eps)$. Fix a point $x\in M$. For any $y \in B_x(\eps)$ the Riemannian distance, $d(x,y)$, from $x$ to $y$ satisfies $d(x,y)<\eps$.  We let \[
U_\eps = \{ (x,y) \in M\times M \;  | \;  y \in B_x(\eps)\;\}.
\]

Let
\[
G(t,x,y) =(4\pi t)^{-n/2} e^{-\frac{d^2(x,y)}{4t}}
\]
be the direct analogue of the Euclidean heat kernel on $M$ which belongs to $C^\infty (\mathbb{R}^+\times U_\eps)$. Set
\[
u(t,x,y)=  u_0(x,y) + \ldots + u_k(x,y) \,t^k
\]
where the functions $u_i(x,y)$ are to be determined. Define
\[
S_k(t,x,y) =  G(t,x,y) \; e^{\frac 12 (f(x)+f(y))}  \; u(t,x,y).
\]

Recalling that
\[
\Delta(h\cdot g)= (\Delta h) g + h \Delta g - 2\< \n h, \n g\>
\]
and for $(x,y)\in U_\eps$ we compute
\begin{equation*}
\begin{split}
 \left( \frac{\p}{\p t} + \Delta_{f,y} \right)  S_k  = &  \left( \frac{\p}{\p t} + \Delta_y \right)  (G u) \cdot  e^{\frac 12 (f(x)+f(y))}+
 G u \cdot \Delta_y  (e^{\frac 12 (f(x)+f(y))})\\
 & - 2 \< \n_y   (G u ), \n_y   (e^{\frac 12 (f(x)+f(y))}) \> + \<  \n_y   f,  \n_y   (G u \,e^{\frac 12 (f(x)+f(y)))} \>
\end{split}
\end{equation*}
where the drifting Laplacian, Laplacian and gradient are taken with respect to the $y$ variable.  Using (3.8) of \cite{Ro} we get
\begin{equation} \label{heat1}
\begin{split}
 \left( \frac{\p}{\p t} + \Delta_{f,y} \right) & S_k  =  G \,e^{\frac{1}{2} (f(x)+f(y))} \left[ u_1 + \ldots + k t^{k-1} u_k + \frac{r}{2 t}\, \frac{D'}{D} (u_0 + \ldots +  t^{k} u_k) \right.  \\
&  \qquad \qquad  +\left. \frac{r}{t} \left(\frac{\p u_0}{\p r} + \ldots + t^k\frac{\p u_k}{\p r}\right) + \Delta_y u_0 + \ldots + t^k \Delta_y u_k \right]     \\
 & + G \, (u_0 + \ldots + t^k u_k) \; e^{\frac 12 (f(x)+f(y))} \cdot \left[ \frac 12 \, \Delta_y f (y) + \frac 14 \, |\n_y f(y)|^2 \right]
\end{split}
\end{equation}
where
\[
D=\text{det}(d \exp_x)
\]
is the determinant of the Riemannian volume form centered at $x$, and $D' = \p D/ \p r$ is its derivative with respect to the radial function $r(y)=d(x,y).$

To obtain the parametrix we will choose $u_i$ inductively such that the coefficient of $t^i$ vanishes for $-1\leq i \leq k-1$. The coefficient of $t^{-1}$ will vanish, if we set
\[
\frac{r}{2}\, \frac{D'}{D} \, u_0 + r \frac {\p u_0}{\p r} =0.
\]
This first order differential equation has a smooth solution for $r<\eps$ given by
\[
u_0(x,y) = D^{-1/2}(y)
\]
which is a smooth function on $U_\eps$ independently of $f$ and satisfies
\[
u_0(x,x)=1.
\]

For $t^{i-1}$ we get the equation
\[
i \, u_i + \frac{r}{2}\, \frac{D'}{D} \, u_i + r \frac {\p u_i}{\p r} + \Delta_y u_{i-1} +  \left[ \frac 12 \, \Delta_y f (y) + \frac 14 \, |\n_y f(y)|^2 \right]  \cdot u_{i-1} =0.
\]
Letting $x(s)$ be the unit speed geodesic from $x$ to $y$ for $s\in [0,r]$ with $x(0)=x$ and  $x(r)=y$, we may obtain a solution to the above equation from the integral equation
\begin{equation} \label{u_i}
\begin{split}
u_i&(x,y)  = - r^{-i}(x,y) D^{-1/2}(y) \;  \left[  \int_0^r D^{1/2}(x(s))  \cdot   (\Delta_{x(s)} u_{i-1}) (x,x(s)) \cdot s^{i-1} ds \right.\\
&+ \left.   \int_0^r D^{1/2}(x(s)) \left( \frac 12 \, \Delta f (x(s)) + \frac 14 \, |\n f(x(s))|^2 \right)  u_{i-1}(x,x(s)) \cdot s^{i-1} ds \right].
\end{split}
\end{equation}
We take this opportunity to correct a small misprint in \cite{Ro} (3.12); in the above $x$ stays fixed, and $y$ varies along the geodesic from $x$. Observe that the functions $u_i$ are smooth on $M\times M$.

As a result,
\begin{equation} \label{heat2}
\begin{split}
 \left( \frac{\p}{\p t} + \Delta_{f,y} \right)   S_k  = &   G(t,x,y) \, t^k  \, e^{\frac{1}{2} (f(x)+f(y))}  \\
 &  \cdot  \left[  \Delta_y u_k (x,y)   +  \left(\frac 12 \, \Delta f (y) + \frac 14 \, |\n f(y)|^2\right) \cdot u_k(x,y)\right].
\end{split}
\end{equation}

Define
\begin{equation*}
\eta(x,y)= \left\{
\begin{array}{cc}
   1 & \text{on} \ \ U_{\eps/2} \\
0 &  \text{on} \ \ M\times M \setminus U_{\eps}
\end{array}\right.
\end{equation*}
to be a smooth function with bounded first and second order derivatives. Then we extend $S_k$ to $M\times M$ setting
\begin{equation} \label{Hk}
h_k(t,x,y) = \eta(x,y) \cdot S_k(t,x,y) = \eta(x,y) \, G(t,x,y) \,e^{\frac 12 (f(x)+f(y))} \sum_{i=1}^k u_i(x,y)\, t^i.
\end{equation}
Each $h_k(t,x,y)$ is a smooth function on $(0,\infty) \times M \times M$.  Moreover, for $k>n/2$, $h_k(t,x,y)$ is a local parametrix of the operator $\frac{\p}{\p t} +\Delta_f$ as it satisfies the following two properties
\begin{lem}
\begin{equation*}
\begin{split}
(i) & \quad \frac{\p}{\p t} h_k +\Delta_{f,y}  h_k  \in C^0([0,\infty)\times M \times M) \\
(ii) &  \quad \lim_{t \to 0} \int_M  h_k(t,x,y) g(y) \,  d\mu_f (y) = g(x) \quad  \text{for any} \quad  g\in L^2_f(M).
\end{split}
\end{equation*}
\end{lem}

\begin{proof}
For $(i)$ we need to show that $\frac{\p}{\p t} h_k +\Delta_{f,y}  h_k$ extends to $t=0$. This is true for $M\times M \setminus U_\eps$ since $h_k\equiv0$. On $U_{\eps/2}$ \eref{heat2} holds, and the right set tends to $0$ as $t\to 0$ for $k>n/2$ and $f$ smooth. On $ U_{\eps} \setminus U_{\eps/2}$
\begin{equation*}
\begin{split}
\left(  \p /\p t  +\Delta_{f,y}  \right) h_k  & =  \eta \, \left(  \p /\p t  +\Delta_{f,y}  \right) S_k -2  \langle d\eta, d S_k \rangle + (\Delta_{f,y} \eta) \, S_k\\
& = (4\pi t)^{-n/2} e^{-\frac{d^2(x,y)}{4t}} \phi(t,x,y)
\end{split}
\end{equation*}
where $\phi$ is a smooth function on $(0,\infty) \times M \times M$ and has a pole of order at most $t^{-1}$. Since $d\geq \eps/2$ on this set, we can also extend $\frac{\p}{\p t} h_k +\Delta_{f,y}  h_k$ by zero to $t=0$.

Note that for $k>l+n/2$,  $\left(  \p /\p t  +\Delta_{f,y}  \right) h_k \in C^l([0,\infty)\times M \times M)$ .

For $(ii)$ we can show that for any $w\in L^2(M)$ (which is equivalent to $g=e^{\frac 12 f} w \in L^2_f$)
\begin{equation*}
\begin{split}
\lim_{t \to 0} & \int_M  h_k(t,x,y) e^{\frac 12 f(y)} w(y)  \,  d\mu_f (y) \\
& = e^{\frac 12 f(x)} \lim_{t \to 0} \sum_{i=0}^k t^i \cdot  \int_M G(t,x,y) \, \eta(x,y) \,  u_i(x,y) \, w(y) \; dy \\
& =  e^{\frac 12 f(x)} \lim_{t \to 0} \sum_{i=0}^k t^i  \,  u_i(x,x) \, w(x)  \\
&= e^{\frac 12 f(x)} w(x)
\end{split}
\end{equation*}
since $ G(t,x,y)$ is the heat kernel in $\mathbb{R}^n$ and $\eta(x,x) = u_0(x,x)=1$.

Therefore, for any $g\in L^2_f(M)$ the result follows.
\end{proof}

\begin{lem}
Let
\[
H_f(t,x,y)= h_k(t,x,y) - Q_k * h_k(t,x,y)
\]
$Q_k = \sum_{\lambda =1}^\infty (-1)^{\lambda +1} (\left( (\p/\p t +\Delta_{f,y})(h_k)\right)^{*\lambda}$. Then $H_f(t,x,y) \in C^\infty((0,\infty)\times M \times M)$, it is independent of $k$ for $k> 2+ n/2$ and it is the heat kernel of the semigroup $e^{-t\Delta_f}$.
\end{lem}
The Lemma follows in exactly the same way as \cite{Ro}*{Theorem 3.22}, and simply relies on estimates of convolution operators which give that $|Q_k * h_k| \leq C\cdot t^{k+1-n/2}$.

\begin{proof}[Proof of Theorem \ref{thmhtr}]
Recall the definition that $A(t) \sim \sum_{i=0}^\infty a_i t^i$ as $t \downarrow 0$, if for all $k\geq K_o$
\[
\lim_{t\to 0}\frac{1}{t^k} \left[A(t) - \sum_{i=K_o}^k a_i t^i \right] =0.
\]

The estimate \eref{htr1} is now a direct consequence of the parametrix method and is due to  the fact that $h_k(t,x,x)=  \sum_{i=0}^k u_i(x,x) \, t^i$ and the term  $ Q_k * h_k(t,x,x)$ is at most of order $t^{k+1-n/2}$ for $k$ sufficiently large (see \cite{Ro}*{Proposition 3.23} for all the details).

For the second part of Theorem \ref{thmhtr} we first note that under these hypotheses, by Corollary \ref{th-weyl} and the fact that the eigenfunctions are an orthonormal basis of $L^2 (M, e^{-f} d\mu)$, the heat kernel is trace class. By the heat trace formula and the asymptotic expansion \eref{htr1},
\begin{equation*}
\begin{split}
\tr e^{-t\Delta_f} =  \int_M H_f(t,x,x) \,  d\mu_f(x)  =\sum_k e^{\lambda_k t} &\sim (4 \pi t)^{-n/2} \sum_{i=0}^\infty t^i \, \int_M u_i(x,x) e^{f(x)}\, d\mu_f(x)\\
&=(4 \pi t)^{-n/2} \sum_{i=0}^\infty t^i \, \int_M u_i(x,x)  \, d\mu (x).
\end{split}
\end{equation*}
\end{proof}

Note that for $f=0$ the coefficient functions $u_i$ corresponding to the heat kernel of the Laplacian over the manifold are inductively given by
\begin{equation} \label{u0_i}
u_{0,i}(x,y) = - r^{-i}(x,y) D^{-1/2}(y) \;   \int_0^r D^{1/2}(x(s))  \cdot   (\Delta_{x(s)} u_{i-1}) (x,x(s)) \cdot s^{i-1} ds
\end{equation}
and denote
\begin{equation} \label{a0_i}
a_{0,i}  =  \int_M u_{0,i}(x,x) \, d\mu(x).
\end{equation}
In this case, we also have   $u_{0,0}(x,x)=1$  and  as we have previously mentioned $u_{0,1}(x,x) = \frac 16 \, K(x)$, where $K(x)$ is the scalar curvature of the manifold at the point $x$.

For the weighted case, the definition of the $u_i$ in \eref{u_i} gives
\begin{equation*}
\begin{split}
a_0 &= \Vol(M) \\
a_1 & = \int_M u_{0,1}(x,x)  \, d\mu (x)  -  \int_M \left(\frac 12 \, \Delta f(x) + \frac 14 \, |\n f(x)|^2  \right)\, d\mu (x) \\
 & \ = \int_M \frac 16 \, K(x)  \, d\mu (x)  -  \int_M   \frac 14 \, |\n f(x)|^2  \, d\mu (x).
\end{split}
\end{equation*}
These coincide with the coefficients of $(4\pi t)^{-n/2}$ and  $(4\pi t)^{-n/2} t $ that we saw in Corollary \ref
{httr_drift} for less smooth $f$.

Using these estimates we can now prove the isospectrality results in Theorem \ref{isosp}
\begin{proof}[Proof of Theorem \ref{isosp}]F
By our assumption,
\[
\tr e^{-t\Delta_{f_1}}=\tr e^{-t\Delta_{f_2}} = \sum_k e^{\lambda_k t}.
\]
As a result, the leading term in \eref{htr2} corresponding to $i=0$ must have the same exponent in $t$, giving us $n=m$.  Moreover, the coefficients $a_0$ must coincide, giving us $\Vol(M)= \Vol(N)$.

Next we assume that $M$ and $N$ are orientable surfaces and that the weight functions have equal Dirichlet norms.  By isospectrality, the coefficients $a_1$ must also coincide, so that
\[
\int_M    K_M(x)   \, d\mu_M(x)   = \int_N  K_N (x) \, d\mu_N(x).
\]
By the Gauss-Bonnet theorem, $M$ and $N$ have identical Euler characteristic,
\[
\chi(M)=\chi(N).
\]
The result follows, since two compact oriented surfaces with the same Euler characteristic are diffeomorphic.

The compactness of the set of isospectral drifting Laplacians with smooth weight function follows immediately from the corresponding result for Schr\"odinger operators proven by J. Br\"uning, Theorem 3 of  \cite{bruning}.
\end{proof}

We would like to finish our paper with the following two remarks.

\begin{remark}
The unitary equivalence, of the drifting Laplacian to the Schr\"odinger operator $\Delta + V$ with $V = \frac{1}{2} \Delta f + \frac{1}{4} | \nabla f|^2$ for $V\in L^\infty$ implies that  the heat kernel of $\Delta + V$, $H_V(t,x,y)$,  is related to the heat kernel of the drifting Laplacian, $H_f(t,x,y)$ by the following formula
\[
H_f(t,x,y)=H_V(t,x,y) e^{\frac 12 (f(x)+f(y)\,)}.
\]

Note, that the corresponding parametrix for $\p/\p t + \Delta +V $ would now be $h_k(t,x,y)= \eta(x,y) \, G(t,x,y) \, u(t,x,y)$ with the same $u_i$ as in \eref{u_i}
\end{remark}

\begin{remark}
As we have seen in the proof of Corollary \ref{th-weyl}, $a_0$ determines Weyl's law for the eigenvalues of the operator and in the case of the drifting Laplacian this law is independent of the weight function $f$.

This is also illustrated in the simple case of a constant function $f$. In this setting the eigenvalues of the drifting Laplacian coincide with the eigenvalues of the Laplacian on a compact manifold even though the weighted volume of manifold is different from its Riemannian volume. As a result $\Delta_f$ and $\Delta$ have the same heat trace, independently of $f$. What is fairly surprising is the fact that Weyl's asymptotic formula is independent of the function $f$ for any $f$ satisfying $\frac{1}{2} \Delta f + \frac{1}{4} | \nabla f|^2 \in L^\infty$.
\end{remark}

\end{document}